\documentclass[11pt,a4paper]{article}

\usepackage{amsmath,amsfonts,amssymb,latexsym,graphics,epsfig,url}
\usepackage[dvipsnames]{xcolor}
\usepackage[normalem]{ulem}
\usepackage{amsthm}
\usepackage[english]{babel}
\usepackage{graphicx,tikz}
\usetikzlibrary{decorations.markings}
\usepackage{soul}
\usepackage{cancel}
\usepackage[utf8]{inputenc}
\usepackage[T1]{fontenc}
\usepackage{hyperref} 
\hypersetup{colorlinks,
    linkcolor=blue, 
    anchorcolor=blue,
    citecolor=blue}
\tikzstyle{directed}=[postaction={decorate,
	decoration={markings,mark=at position 0.6 with {\arrow{latex}}}
}]
\tikzstyle{vertex}=[circle, draw, inner sep=0pt, minimum size=3pt]
\newcommand{\vertex}{\node[vertex]}

\newtheorem{theorem}{Theorem}[section]
\newtheorem{lemma}[theorem]{Lemma}

\newtheorem{proposition}[theorem]{Proposition}

\newtheorem{example}[theorem]{Example}

\DeclareMathOperator{\tr}{tr}
\DeclareMathOperator{\spec}{sp}

\def\Q{\ns Q}

\def\vec0{\mbox{\boldmath $0$}}

\def\A{\mbox{\boldmath $A$}}

\def\I{\mbox{\boldmath $I$}}

\def\I{\mbox{\boldmath $I$}}

\def\Q{\mbox{\boldmath $Q$}}

\def\AA{\mbox{$\mathcal{A}$}}

\def\blue{\textcolor{blue}}

\begin{document}

\title{On the Spectra of Generalized Token Graphs of Stars
\thanks{The research of the first and third authors was funded by AGAUR from the Catalan Government under project 2021SGR00434 and MICINN from the Spanish Government under project PID2020-115442RB-I00.
The third author's research is also supported by a grant from the Universitat Polit\`ecnica de Catalunya, reference AGRUPS-2025.
B. Martínez-Barona acknowledges financial support from the Mexican Government through the SECIHTI Postdoctoral Fellowship.}}

\author{C. Dalf\'o$^a$, M. A. Fiol$^b$, B. Mart\'inez-Barona$^c$, X. Song$^d$\\ 
$^a${\small Departament de Matem\`atica, Universitat de Lleida} \\
		{\small Igualada (Barcelona), Catalonia, {\tt cristina.dalfo@udl.cat}}\\
$^b${\small Departament de Matem\`atiques, Universitat Polit\`ecnica de Catalunya} \\
    	{\small Barcelona Graduate School, Institut de Matem\`atiques de la UPC-BarcelonaTech}\\
    	{\small (IMTech) Barcelona, Catalonia, {\tt miguel.angel.fiol@upc.edu}}\\
$^c$ {\small Department of Applied Mathematics and Systems, Universidad Autónoma Metropolitana}\\ 
        {\small Cuajimalpa, Mexico City, Mexico, {\tt berenice.barona@cua.uam.mx}}\\
$^d${\small School of Humanities and Fundamental Sciences, Shenzhen University of Information}\\
{\small Technology, Shenzhen, Guangdong, China, \texttt{xdsong@suit-sz.edu.cn}}
}
\date{}
\maketitle

\begin{abstract}
In the generalization of token graphs (or supertoken graphs), the vertices of a $k$-token graph of a graph $G$ correspond to $k$ distinguishable or indistinguishable tokens placed on $k$ different or equal vertices of $G$. In several cases, we provide the whole spectrum of a supertoken graph of a star by analyzing its regular partitions.
\end{abstract}




\section{Preliminaries and definitions}


Given some integers $n,k$, we denote by $C^n_k$ (with $k\le n$) and $C\!R^n_k$, respectively, the sets of combinations and multisets with repetition of $n$ elements taken $k$ at a time. Recall that the cardinalities of both sets are
$|C^n_k|=\binom{n}{k}$ and $|C\!R^n_k|=\binom{n+k-1}{k}$.

Let $G=(V,E)$ be a graph on $n$ vertices. Let $k$ and $s$ be positive integers such that $s\le k$.
Then,  the {\it $k$-supertoken} $F_{k}^s(G)$ [or $F^s_{k\times 1}(G)$] is the graph in which each vertex corresponds to a distribution of $k$ equal [or different] tokens between the $n$ vertices of $G$, and in such a way that no vertex can receive more than $s$ tokens.
Thus, every vertex of the $k$-supertoken corresponds to a multiset of unlabeled [or labeled] vertices of $G$.
Moreover, two vertices of $F_{k}^s(G)$ [or $F^s_{k\times 1}(G)$] are adjacent if the symmetric difference of their corresponding multisets 
is a pair of adjacent vertices $u,v\in V$. In other words, each vertex of the $k$-supertoken  is defined by the multiset of vertices of $G$ having a token, and we can go from a vertex to an adjacent one in $F_{k}^s(G)$ [or $F^s_{k\times 1}(G)$] by simply moving, in $G$, one token to a feasible adjacent vertex.


Supertoken graphs have been widely applied across computer science, physics, chemistry, and related fields. 
For example, moving tokens along the edges of a given graph to reach a final configuration is a class of reconfiguration problems in computer science; see Bonnet, Miltzow, and Rz\k{a}\.{z}ewski \cite{BMR2018}. In physical and chemical applications, a class of supertoken graphs is related to the exchange Hamiltonians in quantum mechanics; see Audenaert, Godsil, Royle, and Rudolph \cite{agrr07}. Besides, the minimum cycle basis construction of a class of supertokens may be used to confirm that state-dependent coupling of automata in such a way that it does not violate the principle of microscopic reversibility, see Hammack and Smith \cite{HaSm17}.

Observe that 
\begin{equation}\label{eq:subset1}
    F^1_k(G)\subseteq F_k^2(G)\subseteq \cdots \subseteq F_k^k(G)
\end{equation}
and 
\begin{equation}\label{eq:subset2}
    F^{1}_{k\times 1}(G)\subseteq F^{2}_{k\times 1}(G)\subseteq \cdots\subseteq F^{k}_{k\times 1}(G).
\end{equation}

In Table \ref{taula:casos-part}, we summarize all the cases of token (or supertoken) graphs considered in this paper.

\begin{table}[t]
\begin{center}
\begin{tabular}{|c||c|c|c|c|c|}
  \hline
  $s\setminus$ {\rm tokens} & all equal & all different                                      \\[.1cm]
  \hline\hline
  $1$              & $F_k^1(G)=F_k(G)$       & $F_{k\times 1}^1(G)$                                   \\[.1cm]
                   & Token graph or symmetric power           &                                                           \\[.1cm]
  \hline
  $1<s<k$          & $F_k^s(G)$          & $F_{k\times 1}^s(G)$                                   \\[.1cm]
  \hline
  $k$              & $F_k^k(G)$          & $F_{k\times 1}^k(G)=G\Box\stackrel{(k)}{\cdots}\Box G$ \\[.1cm]
                   & Reduced power                            & Cartesian product                                          \\[.1cm]
  \hline
\end{tabular}
\caption{Particular cases of $k$-supertoken graphs.}
\label{taula:casos-part}
\end{center}
\end{table}

\section{The token graphs $F_k(G)$}
\label{sec:token}
In this case, we deal with the standard $k$-token graph, so the $k$ tokens are indistinguishable, and the maximum number of tokens per vertex is one.
If $G$ has $n$ vertices and $m$ edges, then the token graph $F^{1}_{k}(G)=F_k(G)$
has order $\binom{n}{k}$ and size $\binom{n-2}{k-1} m$. 
The graph $F_k(G)$ is the known {\em $k$-th symmetric power} of $G$ (see Audenaert, Godsil, Royle, and Rudolph~\cite{agrr07}), later renamed {\em $k$-token graph} of $G$ by Fabila-Monroy, Flores-Pe\~{n}aloza, Huemer, Hurtado, Urrutia, and Wood~\cite{ffhhuw12}. 

In particular, when $G$ is the complete graph $K_n$, the token graph $F_k(K_n)$ is the distance-regular graph known as the {\em Johnson graph} $J(n,k)$, which is closely related to some issues of coding theory.
See, for instance, Godsil \cite{Go93}. 
Some properties of $J(n,k)$ are the following:
\begin{itemize}
    \item 
    $J(n,k)$ has $\binom{n}{k}$ vertices and 
    $\binom{n-2}{k-1}\frac{n(n-1)}{2}=\frac{1}{2}k(n-k)\binom{n}{k}$ edges.
    \item 
    $J(n,k)$ has diameter $D=\min\{k,n-k\}$.
     \item 
    The eigenvalues of $J(n,k)$
    are $\lambda_j=(k-j)(n-k-j)-j$ with multiplicities $m_j=\binom{n}{j}-\binom{n}{j-1}$ for $j=0,1,\ldots,D$.
\end{itemize}

More results on token graphs can be found in papers by Lea\~{n}os and Ndjatchi \cite{LN2021}, 
Lea\~{n}os and Trujillo-Negrete \cite{LT2018}, Dalf\'o, Duque, Fabila-Monroy, Fiol, Huemer,  Trujillo-Negrete, and Zaragoza Mart\'inez \cite{ddffhtz2021}, and Barik and Verma \cite{bv24}.

\subsection{The 2-token graph $F_2(G)$ of a star}
\label{subsec:2-token}

In the following, we give some properties of the graphs $F_2(S_n)$, where $S_n$ stands for the star with $n(\ge 2$) leaves (so, with order $n+1$). See Figure \ref{fig:S4+F2(S4)} for $S_4$ and $F_2(S_4)$.

\begin{figure}[t]
    \begin{center}
  \includegraphics[width=8cm]{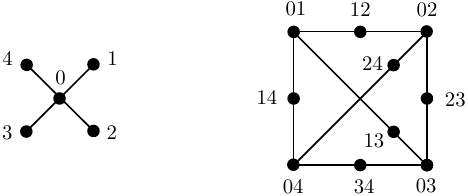}
  \end{center}
  \caption{The star $S_4$ and its 2-token graph $F_2(S_4)$, where a vertex $ab$ corresponds to the set $\{a,b\}$.}
  \label{fig:S4+F2(S4)}
  \end{figure}

Some simple properties of these graphs are the following:
\begin{itemize}
\item[$\circ$]
For $n\ge 3$, $F_2(S_{n-1}) \subset F_2(S_{n}$).

\item[$\circ$]
$F_2(S_n)$ has $\binom{n+1}{2}$ vertices and $n(n-1)$ edges.

\item[$\circ$]
$F_2(S_n)$ is a bipartite graph.

\item[$\circ$]
If $n>3$, then $F_2(S_n)$ has diameter $D=4$.
\end{itemize}

We recall the following result by Cvetkovi\'{c}~\cite{Cv75}, since it is used in the proof of Proposition~\ref{prop:sp-F_2^1(S_n)}.

\begin{theorem}[\cite{Cv75}]
\label{thm:Cv}
If $G$ is an $r$-regular graph with $n$ vertices and $m\big(=\frac{1}{2}nr\big)$ edges, then
$$
P_{S(G)}(z) = z^{m-n} P_{G}(z^2-r),
$$
where $P_{G}(z)$ is the characteristic polynomial of $G$, and $S(G)$ is the subdivison graph of $G$, obtained by adding a new vertex on every edge of $G$.
\end{theorem}

\begin{proposition}
\label{prop:sp-F_2^1(S_n)}
The adjacency spectrum of $F_2(S_n)$, for $n\ge 3$, is
$$
\spec F_2(S_n)=\left\{\sqrt{2(n-1)}^{\,1},\sqrt{n-2}^{\,n-1}, 0^{\,\binom{n}{2}-n},-\sqrt{n-2}^{\,n-1},-\sqrt{2(n-1)}^{\,1}\right\}.
$$
Moreover, the vertex set of $F_2(S_n)$ admits a regular partition.
\end{proposition}

\begin{proof}
Let $V(K_n)=\{1,\ldots,n\}$ and $V(F_2(S_n))=\{(0,1),\ldots,(0,n)\}\cup\{(i,j):i< j$  and $i,j\in [n]\}$, where $[n]=\{1,2,\ldots,n\}$. Denote by $ij$ the new vertex on the edge $ij$ of $K_n$. 
Note that there is a natural one-to-one mapping between $\{1,\ldots,n\}$ and $\{(0,1),\ldots,(0,n)\}$, and between $ij$ and $(i,j)$ for $i\ne j$ and $i,j\in [n]$. Consequently, $F_2(S_n)\cong S(K_{n})$.
From the known characteristic polynomial of the complete graph $K_{n}$, we compute $P_{F_2(S_n)}(z)$.
\begin{align*}
 P_{K_{n}}(z) &= (z+1)^{n-1}(z-n+1),\\
  P_{K_{n}}(z^2-n+1) &= (z^2-n+2)^{n-1}(z^2-2n+2),\\
   P_{F_2(S_n)}(z) &= P_{S(K_{n})}(z) = z^{n(n-1)/2-n}P_{K_{n}}(z^2-n+1)\\
                     &= z^{(n^2-3n)/2}(z^2-n+2)^{n-1}(z^2-2n+2).
\end{align*}
From $P_{F_2(S_n)}(z)$, we obtained the claimed spectrum. Note that, for $n=3$, the multiplicity of the eigenvalue 0 is 0, which means that there is no such eigenvalue.

Moreover, the vertex set of $F_2(S_n)$ admits the regular partition given by the intersection diagram represented in Figure~\ref{fig:quotients}, where we have two kinds of vertices: $n$ vertices corresponding to having one token in a leaf and the other in the center of the star (named vertex 0), 
and $\binom{n}{2}$ vertices with the tokens in the leaves. 
(The different classes of vertices are illustrated for the case $n=4$.)
Some eigenvalues of $F_2(S_n)$ are obtained from the quotient matrix of its corresponding regular partition.
The quotient matrix and its spectrum are
\[
  \Q(F_2(S_n)) = \left(
  \begin{array}{cc}
    0 & n-1 \\
    2 & 0
  \end{array}
  \right),
  \qquad \spec \Q(F_2(S_n))
  =\left\{\pm\sqrt{2(n-1)}^{\,1}\right\}.
\qedhere\]
\end{proof}

\begin{figure}[t]
    \begin{center}
  \includegraphics[width=12cm]{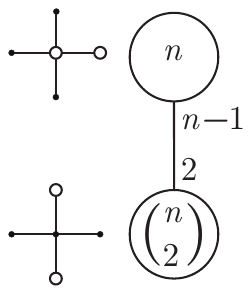}
  \end{center}
  \vskip-13cm
  \caption{Regular partition of $F_2(S_n)$.}
  \label{fig:quotients}
  \end{figure}


\subsection{The $k$-token graph of the star $S_{2k-1}$}
\label{subsec:k-token}
Let $X$ be a set of cardinality $2k -1$. The {\em doubled Odd graph} $2O_k$ on $X$ is the graph 
whose vertices are the $k$-subsets and $(k-1)$-subsets of $X$, where two distinct vertices $u,v$ are adjacent if and only if $u\subset v$ or $v\subset u$.
The double Odd graph $2O_k$ is also known as the {\em middle cube graph} of parameter $k$, here denoted by $Q_{k-1}'$. This is because $2O_k$ is isomorphic to the subgraph of the binary Hamming graph $H(2k-1,2)$, or $(2k-1)$-cube $Q_{2k-1}$, induced by the set of vertices whose binary representation has either $k-1$ or $k$ ones.
%
In \cite{dfm2015}, Dalf\'o, Fiol, and Mitjana studied some of the properties of $Q_{k-1}'$ in the
light of the theory of distance-regular graphs. In particular, they completely determined their eigenvalues and associated eigenvectors. The middle cube graphs are also an example of the so-called `boundary graphs' satisfying a limit case of a property of the diameter of a graph in terms of its spectrum (see Fiol, Garriga, and Yebra \cite{fgy2001}).

Some  properties of the graphs $F_k(S_{2k-1})$ are the following:
\begin{itemize}
\item[$\circ$]
$F_k(S_{2k-1})$ is a $k$-regular graph with $2\binom{2k-1}{k}$ vertices and  $(2k-1)\binom{2k-2}{k-1}$ edges.
\item[$\circ$]
$F_k(S_{2k-1})$ has diameter $D$ and radius $r$ given by $D=r=2k-1$.
\end{itemize}
For an example, see $F_3(S_5)$ in Figure \ref{fig:F3(S5)}.
\begin{proposition}[\cite{bcn1989}]
The $k$-token graph  $F_k(S_{2k-1})$ of the star $S_{2k-1}$ is a bipartite antipodal distance-regular graph with intersection array 
\begin{align*}
\{b_0,b_1,\ldots,b_{2k-2};c_1,c_2,\ldots,c_{2k-1} \} & = \\
 & \hskip -1cm\{k,k-1,k-1,\ldots\ldots,1,1;1,1,\ldots,k-1,k-1,k \}.
\end{align*} 
That is,
\begin{align*}
b_0&=k, \quad b_{2i-1}=b_{2i}= k-i,\quad i=1,\ldots,k-1.\\
c_{2k-1}&=k, \quad c_{2k-2i}=c_{2k-2i-1}= k-i,\quad i=1,\ldots,k-1.
\end{align*}
\end{proposition}
\begin{proof}
This is a direct consequence of the isomorphism between
$F_k(S_{2k-1})$ and the double Odd graph $2O_k$. Indeed, adjacent vertices of $F_k(S_{2k-1})$ are of the form $\{1,2,\ldots, k\}$ and $\{0,2,3,\ldots, k\}$, which in $2O_k$ correspond to the vertices $u=\{1,2,\ldots, k\}$ and $v=\{2,3,\ldots, k\}$.
  The antipodal vertices, at distance $D=2k-1$, are of the form $\{1,2,\ldots, k\}$ and $\{0,k+1,\ldots,2k-1\}$. 
\end{proof}
By considering again the isomorphism between $F_k(S_{2k-1})$ and $2O_k\cong Q_{k-1}'$, we have the following result, cited 
in Brouwer, Cohen and Neumanier \cite[p. 414]{bcn1989}.
For completeness, we prove it using the spectrum of the Johnson graph $J(2k-1,k)$.
\begin{proposition}
\label{propo:F_k(S_2k-1}
The eigenvalues and multiplicites of $F_k(S_{2k-1})$
are
\begin{equation}
\label{eigenvals+mult}
\theta_{j}=\pm 
(k-j), 
\quad m_{j}=\binom{2k-1}{j}-\binom{2k-1}{j-1},\quad j=0,1,\ldots, k-1.
\end{equation}
\end{proposition}
\begin{proof}
Let $\A$ be the adjacency matrix of $F_k(S_{2k-1})$. Then, its distance-2 graph $G_2$ has adjacency matrix $\A_2=\A^2-k\I$. The graph $G_2$ consists of two copies of the Johnson graph $J(2k-1,k)$ with eigenvalues $\lambda_j=(k-j)(k-j-1)-j$ and multiplicities
$m_i=\binom{2k-1}{j}-\binom{2k-1}{j-1}$, for $j=0,1,\ldots, k-1$. Thus, from $\lambda_{j}=\theta_{j}^2-k$, we get
$$
\theta_{j}=\pm\sqrt{\lambda_j+k}=\pm\sqrt{(k-j)(k-j-1)-j+k}=\pm (k-j), \quad j=0,1,\ldots,k-1.
$$
with multiplicities given in \eqref{eigenvals+mult}.
\end{proof}

For instance, for the graph $G=F_3(S_5)$ in Figure \ref{fig:F3(S5)}, we have that $G_2$ consists of two copies of the Johnson graph $J(5,3)$ with diameter $D=k-1=2$. Then, the spectrum of $G$ is
$$
\spec F_3(S_5)=\{\pm 3^1, \pm 2^4, \pm 1^5\}.
$$

\begin{figure}[t]
    \begin{center}
  \includegraphics[width=9cm]{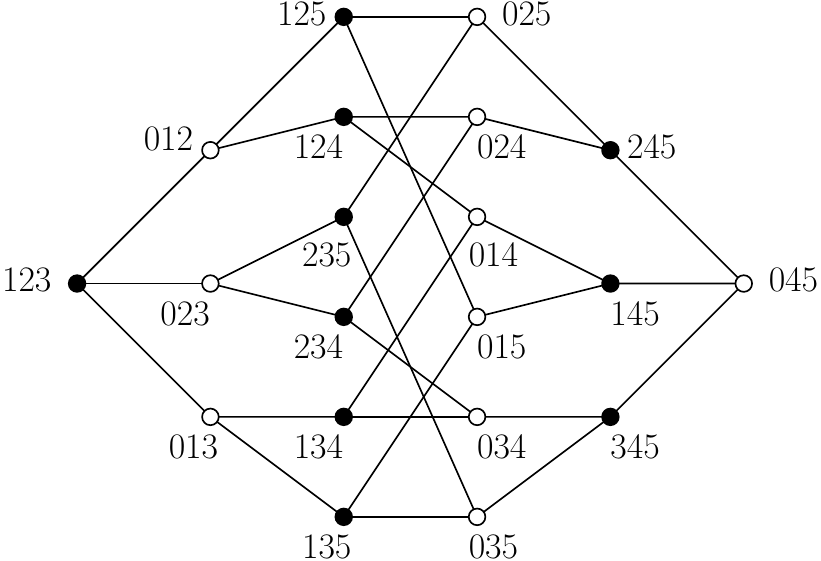}
  \end{center}
  \caption{The 3-token graph of $S_5$ ($S_5$ with central vertex 0).}
  \label{fig:F3(S5)}
  \end{figure}

\section{The supertoken graphs $F_k^s(G)$}
\label{sec:F_k^s(G)}

In this case, the $k$ tokens are indistinguishable, and the maximum number of tokens per vertex is $s$, with $1<s< k$.
The order and size of $F_k^s(G)$ are given in the following result (see Song, Dalf\'o, Fiol, Mora, and Zhang \cite{sdfmz25}).

\begin{proposition}[\cite{sdfmz25}]
Let $G=(V,E)$ be a graph on $n$ vertices. The order and size of the supertoken graph $F_k^s(G)$ are, respectively,
\begin{equation*}\label{nomial}
|V(F_k^s(G))| = \sum_{i=0}^{\lfloor k/(s+1)\rfloor} (-1)^i \binom{n}{i} \binom{n+k-1-i(s+1)}{n-1},
\end{equation*}
and 
\begin{equation*}\label{eq:edge-fks}
|E(F_k^s(G))| = |E| \left(\sum_{i=1}^s i f(n-2,k-i,s) + \sum_{j=1}^{s-1} (s-j)f(n-2,k-s-j,s) \right).
\end{equation*}
\end{proposition}

\subsection{The supertoken graphs $F_3^2(G)$ of a star}

Regarding the star graph, some simple properties of $F_3^2(S_n)$ are the following, where $S_n$ stands for the star with $n(\ge 2$) leaves (so, with order $n+1$).
\begin{itemize}
\item[$\circ$]
$F_3^2(S_n)$ has $N=\binom{n+1}{3}+n^2+n$ vertices. 

\item[$\circ$]
$F_3^2(S_n)$ has $n\left[\binom{n}{2}+2n-1\right]$ edges.

\item[$\circ$]
$F_3^2(S_n)$ is a bipartite graph.

\end{itemize}
\begin{proposition}
\label{propo:F_3^2(S_n)}
    The vertex set of $F_3^2(S_n)$ admits a regular partition, which gives some of its eigenvalues, namely, $0$,$\pm\sqrt{5n-4}$, and $\pm\sqrt{n-2}$.
\end{proposition}

\begin{proof}
    Let $V(S_n)=\{0,1,\ldots,n\}$ with $0$ as the central vertex of $S_n$. Denote by $(u_1,u_2,u_3)$ the vertex of $F_3^2(S_n)$, where $u_1,u_2,u_3\in V(S_n)$ and at most two of them are equal. Let $\pi=(V_1,V_2,\ldots,V_5)$ be a partition of $V(F_3^2(S_n))$, where
\begin{itemize}
\item $V_1=\{(u_1,u_2,u_3):u_1,u_2,u_3\ne 0 ~\text{and $u_1,u_2,u_3$ are distinct}\}$,\quad $|V_1|=\binom{n}{3}$.
\item
$V_2=\{(0,u_2,u_3):u_2,u_3\ne 0 ~\text{and $u_2\neq u_3$}\}$,\quad $|V_2|=\binom{n}{2}$.\item 
$V_3=\{(0,0,u_3):u_3\in[n]\}$,\quad $|V_3|=n$.
\item
$V_4=\{(0,u_2,u_3):u_2=u_3\in[n]\}$,\quad $|V_3|=n$.
\item   
$V_5=\{(u_1,u_2,u_3):u_1{,u_2}\in[n],u_2=u_3\ne u_1\}$,\quad $|V_5|=n(n-1)$.
    \end{itemize}
   
    The quotient graph of $F_3^2(S_n)$ with respect to $\pi$ is shown in Figure \ref{fig:F^2_3quotient}. The quotient matrix, with rows ans columns indexed by $V_1,V_3,V_5,V_2,V_4$, is 
    \begin{equation*}
    \hskip-.75cm 
    \Q_\pi = \left(
  \begin{array}{ccccccc}
    0 & 0 & 0 & 3 & 0 \\
    0 & 0 & 0 & n-1 & 1 \\
    0 & 0 & 0 & 1 & 1 \\
    n-2   & 2   & 2 & 0 & 0 \\
    0   & 1   & n-1 & 0 & 0  
  \end{array}
  \right),
    \end{equation*} 
showing that the stable sets are $V_1\cup V_3\cup V_5$ and $V_2\cup V_4$. The spectrum of this matrix is   
$$
\spec \Q_\pi=\{
0,\pm\sqrt{5n-4},\pm\sqrt{n-2}\}\subset\spec F_3^2(S_n).
$$    
  This completes the proof.
\end{proof}

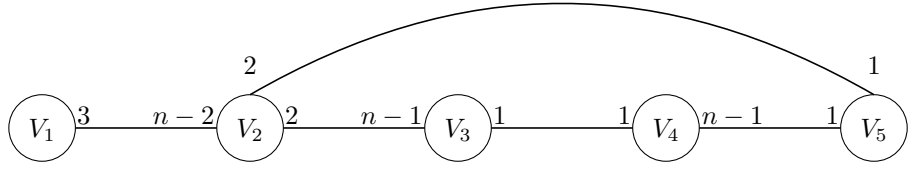
\begin{figure}   
    \begin{center}
    \begin{tikzpicture}[scale=1.1,auto,swap]
    \draw[] (1,1) circle (0.4);
    \node () at (1,0.6) [label=above:{{\small $V_1$}}]{};
    \draw[] (3.5,1) circle (0.4);
    \node () at (3.5,0.6) [label=above:{{\small $V_2$}}]{};
    \draw[] (6,1) circle (0.4);
    \node () at (6,0.6) [label=above:{{\small $V_3$}}]{};
    \draw[] (8.5,1) circle (0.4);
    \node () at (8.5,0.6) [label=above:{{\small $V_4$}}]{};
    \draw[] (11,1) circle (0.4);
    \node () at (11,0.6) [label=above:{{\small $V_5$}}]{};
    \draw (1.4,1) edge [line width=0.6pt] (3.1,1);
    \draw (3.9,1) edge [line width=0.6pt] (5.6,1);
    \draw (6.4,1) edge [line width=0.6pt] (8.1,1);
    \draw (8.9,1) edge [line width=0.6pt] (10.6,1);
    \draw (3.5,1.4) edge [line width=0.6pt,bend left] (11,1.4);
    \node () at (1.5,1.5) [label=below:{\small $3$}]{};
    \node () at (2.7,1.5) [label=below:{\small $n-2$}]{};
    \node () at (4,1.5) [label=below:{\small $2$}]{};
    \node () at (5.2,1.5) [label=below:{\small $n-1$}]{};
    \node () at (6.5,1.5) [label=below:{\small $1$}]{};
    \node () at (8,1.5) [label=below:{\small $1$}]{};
    \node () at (9.3,1.5) [label=below:{\small $n-1$}]{};
    \node () at (10.5,1.5) [label=below:{\small $1$}]{};
    \node () at (3.5,1.4) [label=above:{\small $2$}]{};
    \node () at (11,1.4) [label=above:{\small $1$}]{};
    \end{tikzpicture}
    \caption{The quotient graph of $F_3^2(S_n)$ with respect to the regular partition     $\pi$. 
    }
     \label{fig:F^2_3quotient}
    \end{center}
\end{figure}

\section{The $k$-supertoken graphs $F_k^k(G)$}
\label{sec:F_k^k(G)}

Let $G$ be a graph of order $n$ and size $m$. 
In this case, the $k$ tokens are indistinguishable, and the maximum number of tokens per vertex is $k$.
The supertoken graph $F_k^k(G)$ has order $|V(F_k^k(G))|={|CR^n_{k}|}=\binom{n+k-1}{k}$
and size $|E(F_k^k(G))|=m{\binom{n+k-2}{k-1}}$ (see \cite{sdfmz25}).

This kind of supertoken graph was introduced by Hammack and Smith~\cite{HaSm17}, who named them \emph{reduced power graphs} and denoted them as $G^{(k)}(\cong F_k^k(G))$. They studied some of their structural properties
and provided a construction of minimum cycle bases of $G^{(k)}$.
The 5-supertoken $F_5^5(C_3)$ of the triangle $C_3$ with vertices $a,b,c$ is shown in Figure \ref{fig:F33(C3)}. It has 21 vertices, 45 edges, and the spectrum
$$
\left\{2+\sqrt{8},(1+\sqrt{6})^2,2,\sqrt{3}^2,(-1+\sqrt{2})^2,0^2,2-\sqrt{8},-1^2,(1-\sqrt{6})^2,-\sqrt{3}^2,-2^2,(-1-\sqrt{2})^2\right\}.
  $$
  \begin{figure}[t]
    \begin{center}
  \includegraphics[width=7cm]{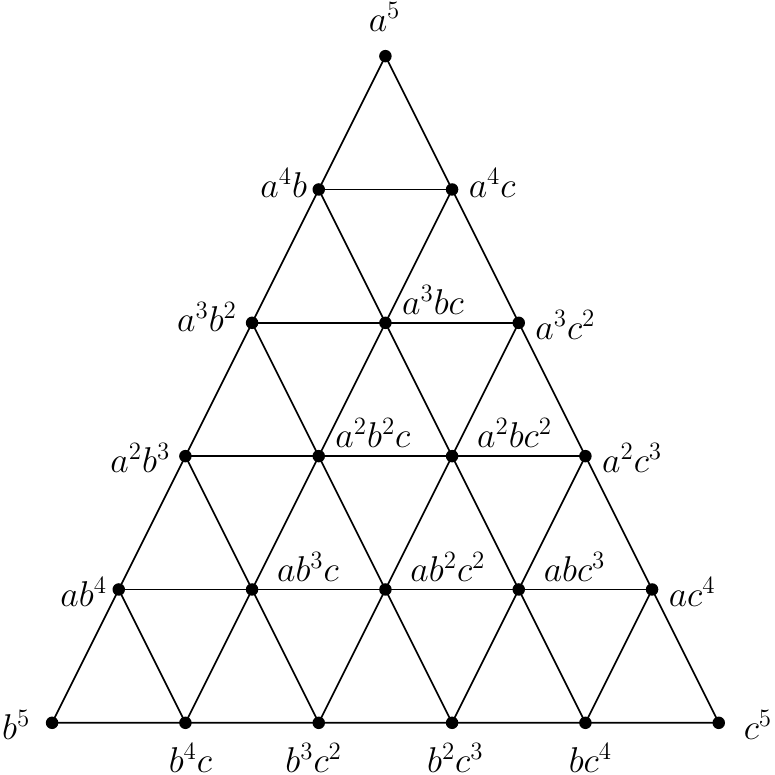}
  \end{center}
  \caption{The 5-supertoken graph $F_5^5$ of the triangle $C_3$ with vertices $a,b,c$.}
  \label{fig:F33(C3)}
  \end{figure}

If $G=K_n$,
then the vertices of the supertoken graph $F_k^k(K_n)$ represent the different possible states of a multiprocessor with $n$ memory modules and $k$ (indistinguishable) processors.

\subsection{The $k$-supertoken graphs $F_2^2(G)$ of a star}

Now, we deal with some properties of the graph $F_2^2(S_n)$, where $S_n$ stands for the star with $n$ leaves (so, with order $n+1$).

Some simple properties of this graph are the following:
\begin{itemize}
\item[$\circ$]
$F_2^2(S_n)$ has $N={\binom{n+2}{2}=}\binom{n}{2}+2n+1$ vertices.

\item[$\circ$]
$F_2^2(S_n)$ has $n^2+n$ edges.

\item[$\circ$]
$F_2^2(S_n)$ is a bipartite graph, and $F_2^1(S_n)\subset F_2^2(S_n)$.

\item[$\circ$]
$F_2^2(S_n)$ has diameter 4.
\end{itemize}

 \begin{figure}[h!]
    \begin{center}
  \includegraphics[width=6cm]{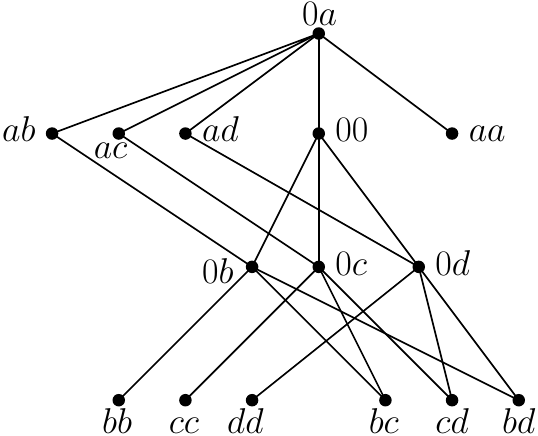}
  \end{center}
  \caption{The supertoken graph $F_2^2(S_4)$.}
  \label{fig:F_2^2(S_4)}
  \end{figure}

\begin{figure}[h!]
    \begin{center}
  \includegraphics[width=14cm]{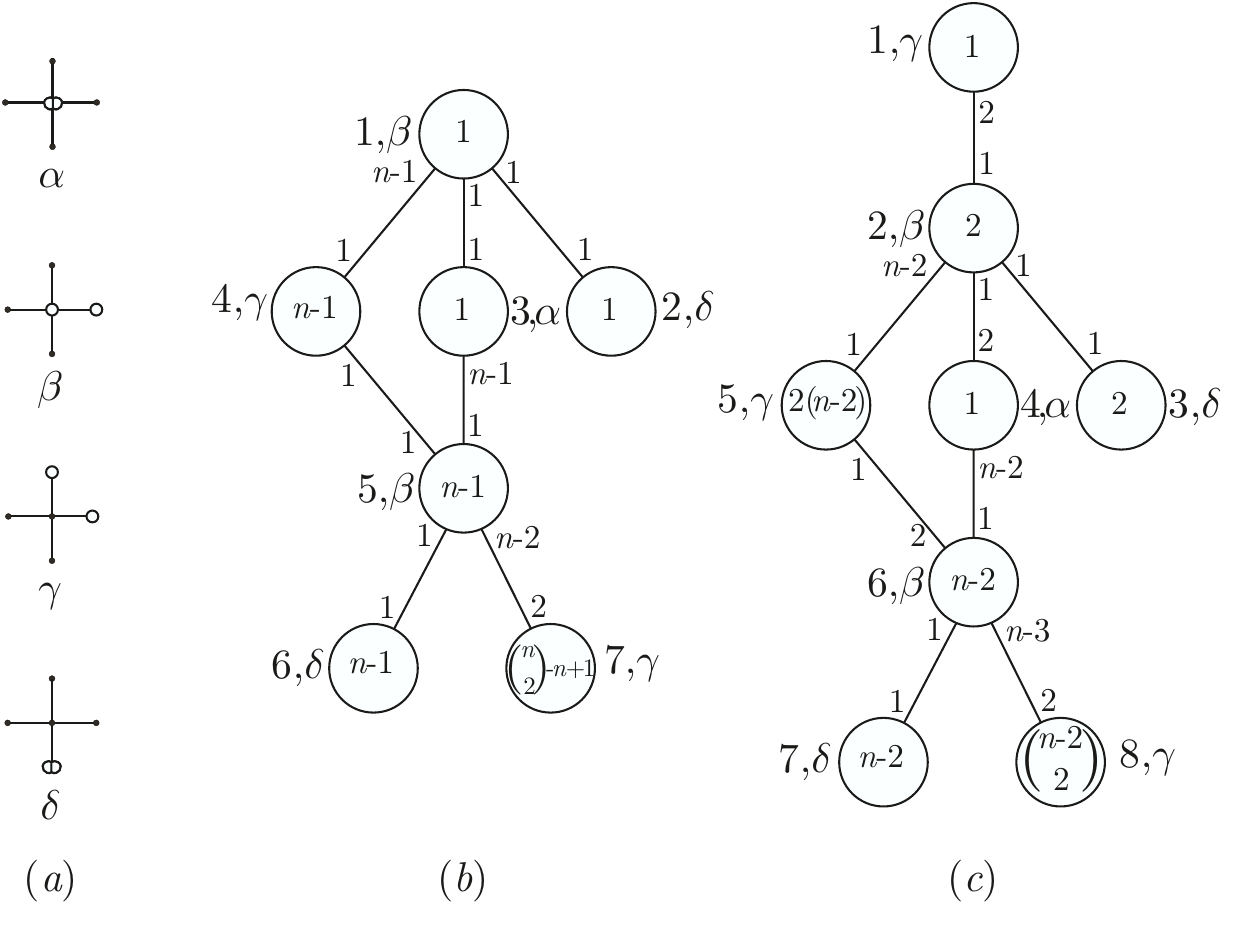}
  \end{center}
  \caption{$(a)$ The different types of vertices in $G=F_2^2(S_n)$. $(b)$ The regular partitions of $G=F_2^2(S_n)$ when $G$ is hung from a vertex $\beta$. $(c)$ The regular partition $G=F_2^2(S_n)$ when $G$ is hung from a vertex $\gamma$. Every vertex has two labels outside itself; the first corresponds to the enumeration of the vertices, and the second indicates the type of vertex ($\alpha$, $\beta$, $\gamma$, or $\delta$). 
  }
  \label{fig:quotients4}
  \end{figure}
  
We recall a very well-known result (see, for example, Godsil \cite{Go93}, or Dalf\'{o} and Fiol~\cite{df17}) since it is used in the proof of Proposition~\ref{prop:sp-F_2^2(S_n)}.
\begin{lemma}[\cite{Go93,df17}]
\label{lemma:CrisMA}
Let $G$ be a (di)graph with adjacency matrix $\A$. Let $\pi=(V_1,\ldots,V_m)$ be a regular partition of $G$, with quotient matrix $\Q$. Then, the number of $k$-walks from each vertex $u\in V_i$ to all vertices of $V_j$ is the
$ij$-entry of $\Q^k$.
\end{lemma}

\begin{proposition}
\label{prop:sp-F_2^2(S_n)}
The adjacency spectrum of $F_2^2(S_n)$ is
$$
\spec F_2^2(S_n)=\left\{\sqrt{3n-1}^{\,1},\sqrt{n-1}^{\,n-1}, 
0^{\,{\binom{n}{2}+1}},-\sqrt{n-1}^{\,n-1},-\sqrt{3n-1}^{\,1}\right\}.
$$
Moreover, the vertex set of $F_2^2(S_n)$ admits a regular partition.
\end{proposition}

\begin{proof}
We have four kinds of vertices in $F_2^2(S_n)$: $n_{\alpha}=1$ vertex $\alpha$, corresponding to having both tokens in the center of the star; $n_{\beta}=n$ vertices $\beta$ with a token in the center and the other in a leave; $n_{\gamma}=\binom{n}{2}$ vertices $\gamma$ with the two tokens in different leaves; and $n_{\delta}=n$ vertices $\delta$ with both tokens in the same leave. See Figure~\ref{fig:quotients4}$(a)$, where the different classes of vertices are illustrated for the case $n=4$.
To obtain regular partitions, we `hung' $F_2^2(S_n)$ first from a vertex $\beta$ (see Figure \ref{fig:F_2^2(S_4)} for the case of $F_2^2(S_4)$), and then from a vertex $\gamma$ (see Figures~\ref{fig:quotients4}$(b)$ and \ref{fig:quotients4}$(c)$). We do not need to hung $F_2^2(S_n)$ from a vertex $\alpha$ or $\delta$, because, in  
Figure~\ref{fig:quotients4}$(b)$, there is a class with only one vertex $\alpha$ and another class with only one vertex $\delta$, whereas in Figure~\ref{fig:quotients4}$(c)$ there is a class with only one vertex $\alpha$.
We write the quotient matrices $\Q_{\beta}$ and $\Q_{\gamma}$ of these two regular partitions, and we obtain their eigenvalues. Both matrices are indexed by the classes $1,2,\ldots, 7$ and  $1,2,\ldots, 8$, respectively. They and their spectra are the following:

\begin{eqnarray*}
  && \hskip-.75cm \Q_{\beta} = 
  \left(
  \begin{array}{ccccccc}
     0 & 1 & 1 & n-1 & 0   & 0 & 0   \\
     1 & 0   & 0 & 0 & 0   & 0 & 0   \\
     1 & 0   & 0 & 0 & n-1 & 0 & 0   \\
     1 & 0   & 0 & 0 & 1   & 0 & 0   \\
     0 & 0   & 1 & 1 & 0   & 1 & n-2 \\
     0 & 0   & 0 & 0 & 1   & 0 & 0   \\
     0 & 0   & 0 & 0 & 2   & 0 & 0
   \end{array}
  \right),
  \quad \spec \Q_{\beta}=\left\{\pm\sqrt{3n-1}^1, \pm\sqrt{n-1}^1, 0^3\right\},\\
  && \hskip-.75cm \Q_{\gamma} = \left(
  \begin{array}{cccccccc}
    0 & 2 & 0 & 0 & 0   & 0   & 0 & 0   \\
    1 & 0 & 1 & 1 & n-2 & 0   & 0 & 0   \\
    0 & 1 & 0 & 0 & 0   & 0   & 0 & 0   \\
    0 & 2 & 0 & 0 & 0   & n-2 & 0 & 0   \\
    0 & 1 & 0 & 0 & 0   & 1   & 0 & 0   \\
    0 & 0 & 0 & 1 & 2   & 0   & 1 & n-3 \\
    0 & 0 & 0 & 0 & 0   & 1   & 0 & 0   \\
    0 & 0 & 0 & 0 & 0   & 2   & 0 & 0
  \end{array}
  \right), \quad \spec \Q_{\gamma}=\left\{\pm\sqrt{3n-1}^1, \pm\sqrt{n-1}^1, 0^4\right\}.
\end{eqnarray*}
The corresponding diagonal element of the $k$-th power of a quotient matrix gives the number of closed walks of length $k=0,1,2,\ldots$ in $F_2^2(S_n)$ rooted at this vertex (see Lemma~\ref{lemma:CrisMA}). More precisely, for every $k=0,1,2,\ldots$, $(\Q_{\beta}^k)_{11}$ gives  the number $n_{\beta}^{(k)}$ of closed walks of length $k$ rooted at a vertex $\beta$,    $(\Q_{\beta}^k)_{33}$ gives that
$n_{\alpha}^{(k)}$ rooted at a vertex $\alpha$, $(\Q_{\gamma}^k)_{11}$ gives that $n_{\gamma}^{(k)}$ of a vertex $\gamma$, and finally $(\Q_{\gamma}^k)_{33}$ gives that 
$n_{\delta}^{(k)}$ of a vertex $\delta$. 
Note that, for each graph $F_2^2(S_n)$, all the above numbers only depend on the eigenvalues in $\spec \Q_{\beta}\cup \spec \Q_{\gamma}$.
Moreover, as we know how many vertices are of each type, namely, $n_{\alpha},n_{\beta},n_{\gamma}$, and  $n_{\delta}$, we can compute the total number of closed $k$-walks, that is,
\begin{equation}
\label{traceA^k}
\textrm{tr} \A^{k}=n_{\alpha}^{(k)}+n\cdot n_{\beta}^{(k)}+{\binom{n}{2}}\cdot n_{\gamma}^{(k)}+n\cdot n_{\delta}^{(k)},\qquad k=0,1,2,\ldots
\end{equation}

Then, once we know all distinct eigenvalues,
that is, those obtained from $\Q_{\beta}$ or $\Q_{\gamma}$, we can obtain the multiplicities by using $\textrm{tr}\A^{k}=\sum_{i=0}^d m_{i} \lambda_i^{k}$ for small values of $k$, where 
$m_i$ is the multiplicity of $\lambda_i$.
For a graph $G$ with order $N$, size $M$, and adjacency matrix $\A$ with eigenvalues $\lambda_0$, $\lambda_1$, \ldots, $\lambda_d$, we have
\begin{equation}
\begin{split}
 k=0 &:\  \textrm{tr}\A^0=N=\sum_{i=0}^d m_{i},\\
 k=2 &:\  \textrm{tr}\A^2=2M=\sum_{i=0}^d m_{i} \lambda_i^2.
\end{split}
\label{eq:multipl}
\end{equation}
It is known that the multiplicity of the largest eigenvalue is $1$ if the graph is connected. Besides, since $F_2^2(S_n)$ is a bipartite graph, its spectrum is symmetric with respect to the eigenvalue $0$. So, the spectrum of $F_2^2(S_n)$ is
$$
\left\{\lambda_0^1=\sqrt{3n-1}^{\,1},\lambda_1^{\,m_1}=\sqrt{n-1}^{m_1}, \lambda_2^{m_2}=0^{\,m_2},-\lambda_1^{\,m_1}=-\sqrt{n-1}^{m_1},-\lambda_0^1=-\sqrt{3n-1}^{\,1}\right\}.
$$
Then, in our case, the system \eqref{eq:multipl} is
\begin{eqnarray*}
  && {\binom{n}{2}}+2n+1=2+2m_1+m_2,\\
  && 2(n^2+n)=2\sqrt{3n-1}^2+2m_1\sqrt{n-1}^2,
\end{eqnarray*}
from where we get $m_1=n-1$ and $m_2={\binom{n}{2}}+1$, as claimed.
\end{proof}

\section{The supertoken graphs $F^{1}_{k\times 1}(G)$}
\label{sec:k-colors-1tokenvertex}

Let $G$ be a graph of order $n$ and size $m$. In this case, the $k$ tokens are different, and the maximum number of tokens per vertex is one.
The supertoken graph $F^{1}_{k\times 1}(G)$ has order $\frac{n!}{(n-k)!}$ and size $mk\frac{(n-2)!}{(n-k-1)!}$.

\begin{proposition}\label{prop:regularF1_1111}
    Let $G$ be a graph with order $n$. 
    Then, the vertex set of $F^{1}_{k\times 1}(G)$ admits a regular partition with quotient graph $F_k^1(G)$.
\end{proposition}
\begin{proof}
    Let $A_1,A_2,\ldots,A_p$ be the vertices of $F_k^1(G)$ with $p=\binom{n}{k}$, 
    and $\AA_i$ be the set of vertices in $F^{1}_{k\times 1}(G)$, where each vertex in $\AA_i$ corresponds to a distinct permutation of the tokens in the vertex $A_i$ for $i=1,\ldots,p$. 
    Note that $V(F^{1}_{k\times 1}(G))=\AA_1\cup\AA_2\cup\cdots\cup\AA_p$. It follows that $\pi=(\AA_1,\AA_2,\ldots,\AA_p)$ is a partition of $V(F^{1}_{k\times 1}(G))$.
    Let $B,B'$ be two vertices in $\AA_i$, where $i=1,\ldots,p$. 
    Assume that $B=(\alpha_1,\alpha_2,\ldots,\alpha_k)$.
    There exists a permutation $\sigma$ of $\{1,\ldots,k\}$
    such that $B'=(\alpha_{\sigma(1)},\alpha_{\sigma(2)},\ldots,\alpha_{\sigma(k)})$.
    It suffices to prove that $B$ and $B'$ have the same number of neighbors in $\AA_j$ for all $j$.
    Suppose first that $j=i$. 
    We find that $B$ and $B'$ have no neighbors in $\AA_i$. 
    Now assume that $j\ne i$.
    Let $C=(\beta_1,\beta_2,\ldots,\beta_k)$ be a neighbor of $B$ in $\AA_j$. Hence, $\alpha_s=\beta_s$ for some $s\in[k]$ and $\alpha_t=\beta_t$ for every $t\in [k]\setminus\{s\}$. Let $f$ be the one-to-one mapping from $\{\alpha_1,\alpha_2,\ldots,\alpha_k\}$ to $\{\beta_1,\beta_2,\ldots,\beta_k\}$ such that $\alpha_s\sim f(\alpha_s)=\beta_s$, and $f(\alpha_t)=\beta_t$ for every $t\in [k]\backslash\{s\}$.    
    Suppose that $\sigma(r)=s$. 
    Therefore, 
    \begin{eqnarray*}
B'=\left(\alpha_{\sigma(1)},\ldots,\alpha_{\sigma(r)},\ldots,\alpha_{\sigma(k)}\right)&\sim&\left(f(\alpha_{\sigma(1)}),\ldots,f(\alpha_{\sigma(r)}),\ldots,f(\alpha_{\sigma(k)})\right)\\
        &=&(\beta_{\sigma(1)},\ldots,\beta_s,\ldots,\beta_{\sigma(k)})\\
        &=&(\beta_{\sigma(1)},\ldots,\beta_{\sigma(r)},\ldots,\beta_{\sigma(k)}) \,=C'\in \AA_j.
    \end{eqnarray*}
   So, with self-explanatory notation, we have proved that
   $$
   B\sim C\ \implies \ B'=\sigma(B) \sim C'=\sigma (C),\quad C,C'\in \AA_j.
   $$
   Thus, $B$ and $B'$ have the same number of neighbors in $\AA_j$ for all $j$.
\end{proof}
\begin{proposition}
\label{propo:eigenvals F_{kx1}^1(G)}
    Let $\lambda_1,\ldots,\lambda_p$ be the adjacency eigenvalues of $F_k^1(G)$. Then $\lambda_i$ is an adjacency eigenvalue of $F^{1}_{k\times 1}(G)$ for $i=1,\ldots,p=\binom{n}{k}$.
\end{proposition}
\begin{proof}
    Let $\A$ be the adjacent matrix of $F_k^1(G)$, and $\pi=(\AA_1,\AA_2,\ldots,\AA_p)$ be a regular partition of $F^{1}_{k\times 1}(G)$ shown in Proposition \ref{prop:regularF1_1111}.
    Consider $B_i\in \AA_i$.
    Note that each vertex in $\AA_i$ corresponds to a distinct permutation of the tokens in the vertex $B_i$ for $i=1,\ldots,p$.
    Thus, if $B_i$ is adjacent to some vertex $B_j\in\AA_j$, then $B_i$ cannot be adjacent to any other vertex in $\AA_j$ in addition to $B_j$.
    Hence, the quotient matrix of $F^{1}_{k\times 1}(G)$ with respect to $\pi$ should be $\A$. Since $\pi$ is a regular partition, the eigenvalues of $\A$ must be the eigenvalues of $F^{1}_{k\times 1}(G)$.
\end{proof}
\begin{example}
    Let $S_3$ be a star with center $0$ and leaves $1,2,3$. 
    Let $\AA_1=\{(0,1),(1,0)\}$, $\AA_2=\{(0,2),(2,0)\}$, $\AA_3=\{(0,3),(3,0)\}$, $\AA_4=\{(2,1),(1,2)\}$, $\AA_5=\{(3,2),(2,3)\}$, and $\AA_6=\{(1,3),(3,1)\}$. 
    It follows that $\pi=(\AA_1,\AA_2,\AA_3,\AA_4,\AA_5,\AA_6)$ is a regular partition of $V(F^{1}_{2\times 1}(S_3))$, and the quotient graph $F^{1}_{2\times 1}(S_3)/\pi$ is shown to the right of Figure \ref{fig:example-of-regular-partition}. 
    Note that the quotient matrix $\Q_{\pi}$ is 
      \begin{equation*}
    \hskip-.75cm 
    \Q_\pi = \left(
  \begin{array}{ccccccc}
    0 & 0 & 0 & 1 & 0  & 1 \\
    0 & 0 & 0 & 1 & 1  & 0 \\
    0 & 0 & 0 & 0 & 1  & 1 \\
    1 & 1 & 0 & 0 & 0  & 0 \\
    0 & 1 & 1 & 0 & 0  & 0 \\
    1 & 0 & 1 & 0 & 0  & 0
  \end{array}
  \right),
    \end{equation*} 
and its spectrum is $\spec \Q_\pi=\{1^2,-1^2,2,-2\}$.
Moreover, we obtain $\spec F_2^1(S_3)=\{1^2,-1^2,$
$2,-2\}=\spec \Q_\pi$ and 
\begin{equation*}
    \spec \Q_\pi\subset\spec F^{1}_{2\times 1}(S_3)=\{0^2, 1^2,-1^2,\sqrt{3}^2,-\sqrt{3}^2,2,-2\}.
\end{equation*}
\end{example}

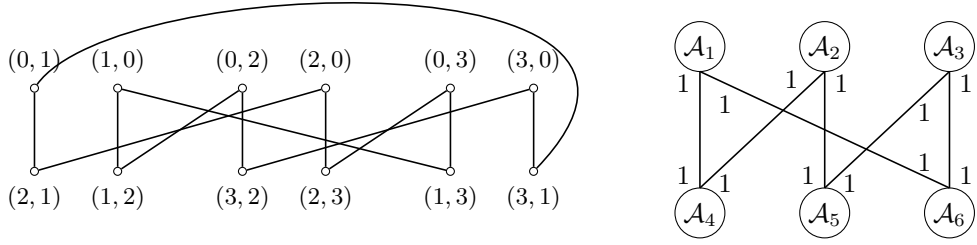
\begin{figure}   
    \begin{center}
    \begin{tikzpicture}[scale=1.1,auto,swap]
    \vertex (01) at (0,0) [label=above:{{\footnotesize$(0,1)$}}]{};
    \vertex (10) at (1,0) [label=above:{{\footnotesize $(1,0)$}}]{};
    \vertex (02) at (2.5,0) [label=above:{{\footnotesize$(0,2)$}}]{};
    \vertex (20) at (3.5,0) [label=above:{{\footnotesize $(2,0)$}}]{};
    \vertex (03) at (5,0) [label=above:{{\footnotesize$(0,3)$}}]{};
    \vertex (30) at (6,0) [label=above:{{\footnotesize $(3,0)$}}]{};
    \vertex (21) at (0,-1) [label=below:{{\footnotesize$(2,1)$}}]{};
    \vertex (12) at (1,-1) [label=below:{{\footnotesize $(1,2)$}}]{};
    \vertex (32) at (2.5,-1) [label=below:{{\footnotesize$(3,2)$}}]{};
    \vertex (23) at (3.5,-1) [label=below:{{\footnotesize $(2,3)$}}]{};
    \vertex (13) at (5,-1) [label=below:{{\footnotesize$(1,3)$}}]{};
    \vertex (31) at (6,-1) [label=below:{{\footnotesize $(3,1)$}}]{};
    \draw [line width=0.6pt] (01)--(21);
    \draw [line width=0.6pt] (01) .. controls (1,1.4) and (8.5,1.6) .. (31);
    \draw [line width=0.6pt] (10)--(12);
    \draw [line width=0.6pt] (10)--(13);
    \draw [line width=0.6pt] (02)--(12);
    \draw [line width=0.6pt] (02)--(32);
    \draw [line width=0.6pt] (20)--(21);
    \draw [line width=0.6pt] (20)--(23);
    \draw [line width=0.6pt] (03)--(13);
    \draw [line width=0.6pt] (03)--(23);
    \draw [line width=0.6pt] (30)--(31);
    \draw [line width=0.6pt] (30)--(32);
    \draw[] (8,0.5) circle (0.3);
    \node () at (8,0.1) [label=above:{{\small $\AA_1$}}]{};
    \draw[] (9.5,0.5) circle (0.3);
    \node () at (9.5,0.1) [label=above:{{\small $\AA_2$}}]{};
    \draw[] (11,0.5) circle (0.3);
    \node () at (11,0.1) [label=above:{{\small $\AA_3$}}]{};
    \draw[] (8,-1.5) circle (0.3);
    \node () at (8,-1.9) [label=above:{{\small $\AA_4$}}]{};
    \draw[] (9.5,-1.5) circle (0.3);
    \node () at (9.5,-1.9) [label=above:{{\small $\AA_5$}}]{};
    \draw[] (11,-1.5) circle (0.3);
    \node () at (11,-1.9) [label=above:{{\small $\AA_6$}}]{};
    \draw (8,0.2) edge [line width=0.6pt] (8,-1.2);
    \draw (9.5,0.2) edge [line width=0.6pt] (9.5,-1.2);
    \draw (11,0.2) edge [line width=0.6pt] (11,-1.2);
    \draw (8,0.2) edge [line width=0.6pt] (11,-1.2);
    \draw (9.5,0.2) edge [line width=0.6pt] (8,-1.2);
    \draw (11,0.2) edge [line width=0.6pt] (9.5,-1.2);
    \node () at (7.8,0.4) [label=below:{\small $1$}]{};
    \node () at (7.8,-0.7) [label=below:{\small $1$}]{};
    \node () at (11.2,0.4) [label=below:{\small $1$}]{};
    \node () at (11.2,-0.7) [label=below:{\small $1$}]{};
    \node () at (8.3,0.15) [label=below:{\small $1$}]{};
    \node () at (8.3,-0.8) [label=below:{\small $1$}]{};
    \node () at (9.1,0.45) [label=below:{\small $1$}]{};
    \node () at (9.3,-0.7) [label=below:{\small $1$}]{};
    \node () at (9.7,0.4) [label=below:{\small $1$}]{};
    \node () at (9.8,-0.8) [label=below:{\small $1$}]{};
    \node () at (10.7,0.05) [label=below:{\small $1$}]{};
    \node () at (10.7,-0.55) [label=below:{\small $1$}]{};
    \end{tikzpicture}
    \caption{The supertoken graph $F^{1}_{2\times 1}(S_3)$ (left) and its quotient graph $F^{1}_{2\times 1}(S_3)/\pi$ $\cong F_{2}^1(S_3)$ (right).}
     \label{fig:example-of-regular-partition}
    \end{center}
\end{figure}

\subsection{The supertoken graphs $F^{1}_{2\times 1}(G)$ of a star}

When the original graph is a star, we have the following result.
 \begin{proposition}
\label{prop:sp-F_2x1^1(S_n)}
The adjacency spectrum of $F_{2\times 1}^1(S_n)$ is
$$ 
\spec F_{2\times 1}^1(S_n)=\left\{\pm \blue{\sqrt{2n-2}}^{\,1},\pm \sqrt{n}^{\,n-1}, 
\pm \sqrt{n-2}^{\,n-1},
0^{\,n^2-3n+2}\right\}.
$$
Moreover, the vertex set of $F_{2\times 1}^1(S_n)$ admits a regular partition with quotient graph $F_{2}^1(S_n)$.
\end{proposition}
\begin{figure}[h!]
    \begin{center}
  \includegraphics[width=14cm]{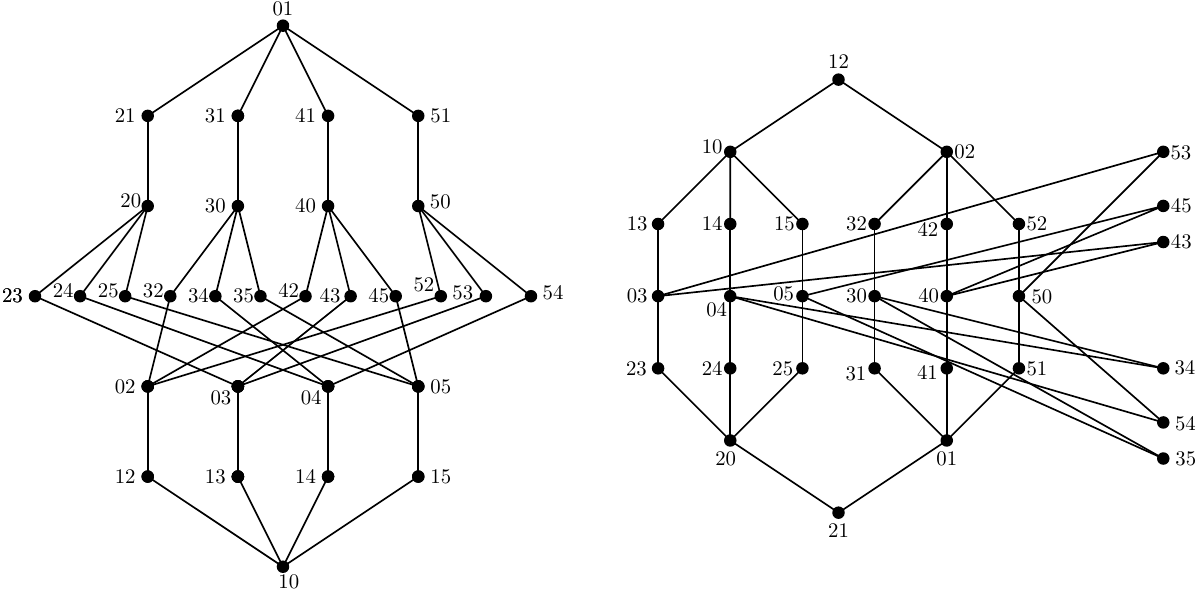}
  \end{center}
  \caption{The supertoken graph $F_{2\times 1}^1(S_5)$ when hanging from the vertices $(0,1)=01$ and $(1,2)=12$, respectively.}
  \label{fig:F_2x1^1(S5)}
  \end{figure}
\begin{figure}[h!]
    \begin{center}
  \includegraphics[width=9cm]{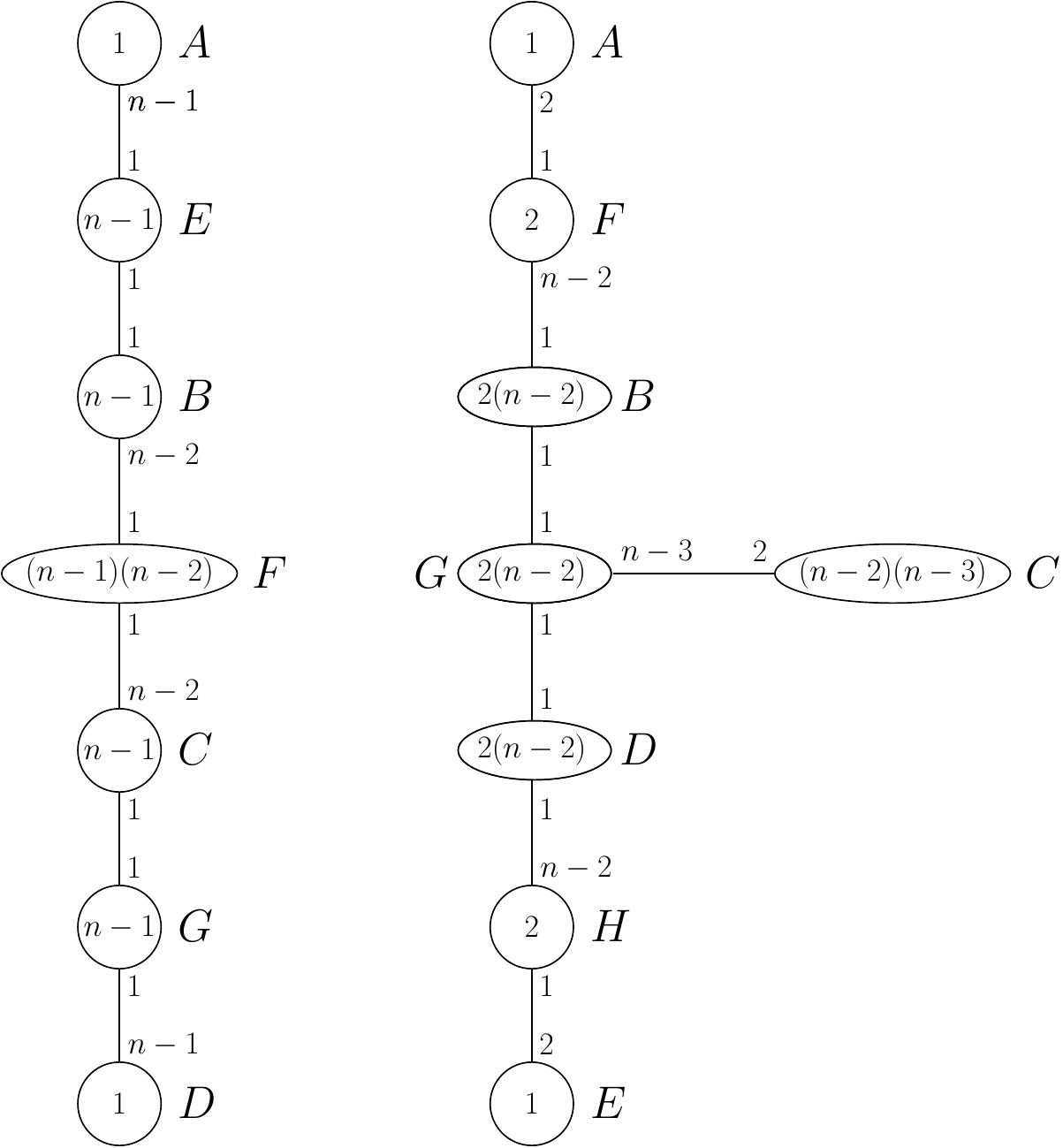}
  \end{center}
  \caption{The quotient diagrams of the regular partitions of the supertoken graph $F_{2\times 1}^1(S_n)$ when hanging from the vertices $(0,a)=0a$ or $(a,0)=a0$, and $(a,b)=ab$, respectively.}
  \label{fig:Quotients}
  \end{figure}
\begin{proof}
The proof goes along the same lines of reasoning as Proposition
\ref{prop:sp-F_2^2(S_n)}.
The graph $F_{2\times 1}^1(S_{n})$ has two classes of vertices. Namely, $n_{0a}=2n$ vertices containing 0 (that is, of the form $(0,a)=0a$ or $(a,0)=a0$), and $n_{ab}=n(n-1)$ vertices of type $(a,b)=ab$, where $a,b\in[n]$ and $a\neq b$. When the graph is hanging from these vertices 
(see Figure \ref{fig:F_2x1^1(S5)} for the case of $F_{2\times 1}^1(S_5)$), 
we get the regular partitions shown in Figure \ref{fig:Quotients}. From them, we see that $F_{2\times 1}^1(S_n)$
is a bipartite antipodal graph with diameter 6. The antipodal vertex of $0a$ is $a0$, and the antipodal vertex of $ab$ is $ba$.
Then, the respective quotient matrices $\Q_{0a}$ and $\Q_{ab}$, indexed by the classes $A,B,C,\ldots$, are:
\begin{equation*} 
    \hskip-.75cm 
    \Q_{0a} = \left( 
  \begin{array}{ccccccc}
    0 & 0 & 0   & 0   & n-1  & 0 & 0     \\
    0 & 0 & 0   & 0   & 1  & n-2 & 0 \\
    0 & 0 & 0   & 0   & 0  & n-2 & 1 \\
    0 & 0 & 0   & 0   & 0  & 0 & n-1 \\
    1 & 1 & 0 & 0   & 0  & 0 & 0\\
    0 & 1 & 1 & 0   & 0  & 0 & 0 \\
    0 & 0 & 1 & 1   & 0 & 0  & 0
  \end{array}
  \right)
  \end{equation*}
  and
  \begin{equation*}
  \Q_{ab} = \left(
  \begin{array}{cccccccc}
    0 & 0 & 0 & 0 & 0   & 2 & 0 & 0      \\
    0 & 0 & 0 & 0 & 0   & 1 & 1 & 0 \\
    0 & 0 & 0 & 0 & 0   & 0 & 2 & 0 \\
    0 & 0 & 0 & 0 & 0   & 0 & 1 & 1 \\
     0 & 0 & 0 & 0 & 0  & 0 & 0 & 2 \\
    1 & n-2 & 0 & 0 & 0 & 0 & 0 & 0\\
    0 & 1 & n-3 & 1 & 0 & 0 & 0 & 0\\
    0 & 0 & 0 & n-2 & 1 & 0 & 0 & 0
  \end{array}
  \right),
    \end{equation*} 
with spectra
\begin{align*}
\spec \Q_{0a}&=\left\{\pm \sqrt{2n-2}^{\,1},\pm \sqrt{n}^{\,1}
,\pm \sqrt{n-2}^{\,1},
0^{\,1}\right\},\\
\spec \Q_{ab}&=\left\{\pm \sqrt{2n-2}^{\,1},\pm \sqrt{n}^{\,1}
,\pm \sqrt{n-2}^{\,1},
0^{\,2}\right\}.
\end{align*}
Then, taking into account that $F_{2\times 1}^1(S_n)$ is bipartite, it must have spectrum
$$
\spec F_{2\times 1}^1(S_n)=\left\{\pm \sqrt{2n-2}^{\,1},\pm \sqrt{n}^{\,m_1},
\pm \sqrt{n-2}^{\,m_2},0^{\, m_3}\right\},
$$
where the multiplicities $m_1$, $m_2$, and $m_3$ satisfy the system
\begin{align}
\tr \A^0 &= 2+2m_1+2m_2+m_3 =n^2+n, \nonumber\\
\tr \A^2 &= 2(2n-2)+2nm_1+2(n-2)m_2=4n^2-4n, \label{system}\\
\tr \A^4 &= 2(2n-2)^2+2n^2m_1+2(n-2)^2m_2=4n^2(n-1),
\nonumber
\end{align}
where $\tr \A^0$ is the number $n^2+n$ of vertices,
$\tr \A^2$ is twice the number $2n^2-2n$ of edges, and  $\tr \A^4$ can be computed by using the numbers of vertices $n_{0a}$ and $n_{ab}$ and the entries $(\Q_{0a}^4)_{11}=n^2-n$ and
$(\Q_{ab}^4)_{11}=2n$, giving $\tr \A^4=n_{0a}(\Q_{0a}^4)_{11}+ n_{ab}(\Q_{ab}^4)_{11}=4n^2(n-1)$. 
Then, the solution of the system \eqref{system} yields $m_1=m_2=n-1$ and $m_3=n^2-3n+2$. This completes the proof.
\end{proof}

Note that, according to Proposition \ref{propo:eigenvals F_{kx1}^1(G)}, all the eigenvalues of $F_2^1(S_n)$ belong to the spectrum of $F^{1}_{2\times 1}(S_n)$.
\section{The supertoken graphs $F^{s}_{k\times 1}(G)$}
\label{sec:F_k^s(G)-colors-dif}

In this case, the $k$ tokens are different, and the maximum number of tokens per vertex is $s$, with $1<s<k$.
For integers $t_1,\ldots,t_i$ satisfying $t_1+\cdots+t_i\le k$, we define $\binom{k}{t_1,\ldots,t_i}=\frac{k!}{t_{1}!\cdots t_{i}!(k-t_{1}-\cdots-t_{i})!}$.

\begin{proposition}[\cite{sdfmz25}]
Let $G=(V,E)$ be a graph on $n$ vertices. The order and size of the supertoken graph $F^{s}_{k\times 1}(G)$ are, respectively,
\begin{eqnarray}\label{eq:nomial2}
|V(F^{s}_{k\times 1}(G))| &= &\sum_{j=0}^{\lfloor\frac{k}{s+1}\rfloor}(-1)^j\binom{n}{j}\sum_{\substack{t_{i_1},t_{i_2},\ldots,t_{i_j}\ge s+1\\ t_{i_1}+t_{i_2}+\cdots+t_{i_j}\le k}}\binom{k}{t_{i_1},\ldots,t_{i_j}}(n-j)^{k-t_{i_1}-\cdots-t_{i_j}},
\end{eqnarray}
and 
\begin{equation}\label{eq:edge-fks2}
|E(F_{k\times 1}^s(G))| = |E|\sum_{r_u=0}^{s-1}\sum_{r_v=0}^{s-1}(r_u+1)\binom{k}{r_u+1,r_v}h(n-2,k-r_u-r_v-1,s).
\end{equation}
\end{proposition}
Specifically, for $j=0$, we define the second sum $\sum_{\substack{t_{i_1},t_{i_2},\ldots,t_{i_j}\ge s+1\\ t_{i_1}+t_{i_2}+\cdots+t_{i_j}\le k}}\binom{k}{t_{i_1},\ldots,t_{i_j}}$ to be equal to $1$ in (\ref{eq:nomial2}), and for $r_u+r_v=k-1$, we define $h(n,0,s)$ to be $1$ in (\ref{eq:edge-fks2}).

Song, Dalf\'o, Fiol, Mora, and Zhang \cite{sdfmz25} showed that 
$F_{k\times 1}^k(G)$ is isomorphic to the Cartesian product $G\Box\stackrel{(k)}{\cdots}\Box G$.
Thus, every eigenvalue of $F_{k\times 1}^k(G)$ is obtained by adding $k$ eigenvalues (not necessarily distinct) of $G$.

Let us now consider the case of $F_{2\times 1}^{2}(S_n)$. This is a bipartite graph with $n^2+2n+1$ vertices, $2n^2+2n$ edges, and diameter 4. For example, since the star $S_n$
has spectrum $\spec S_{n}=\{ \pm\sqrt{n}^1,0^{n-1}\}$
the adjacency spectrum of $F_{2\times 1}^2(S_n)\cong S_n\Box S_n$ is
$$ 
\spec F_{2\times 1}^2(S_n)=\left\{\pm 2\sqrt{n}^{\,1},\pm \sqrt{n}^{\,2n-2}, 
0^{\,n^2-2n+3}\right\}.
$$

\subsection{The supertoken graphs $F^{2}_{3\times 1}(G)$ of a star}

Concerning the supertoken $F_{3\times1}^2(S_n)$, Figure \ref{fig:supertoken(maco)} shows this graph for $n=2$ with 24 vertices and 42 edges. 
In general, the graph $F_{3\times1}^2(S_n)$ has $n^3+3n^2+2n$ vertices and $3n^3+6n^2-3n$ edges. 
\begin{proposition}
\label{propo:F_3x1^2(S_n)}
 As a bipartite graph, the spectrum of $F_{3\times1}^2(S_n)$ contains the eigenvalues
 \begin{equation}
 \begin{split}
 0,\quad \rho=\frac12\sqrt{-6 + 6\sqrt{4n^2 + 4n - 7} + 16n},\quad -\frac12\sqrt{-6 + 6\sqrt{4n^2 + 4n - 7} + 16n}, \\
 \frac12\sqrt{-6 - 6\sqrt{4n^2 + 4n - 7} + 16n},\quad -\frac12\sqrt{-6 - 6\sqrt{4n^2 + 4n - 7} + 16n},
\end{split}
\label{eigenvals:supertoken}
\end{equation}
where $\rho$ is the spectral radius.
\end{proposition}
\begin{proof}
In the case of $n=2$, there are four classes of vertices, whereas for $n\ge 3$, there are five classes of vertices, as illustrated in the quotient diagram of Figure \ref{fig:Quotient2} (Notice that, for $n=2$ the `class $A$' has zero vertices and the quotient matrix also gives the correct eigenvalues).
\begin{figure}[t]
    \begin{center}
  \includegraphics[width=8cm]{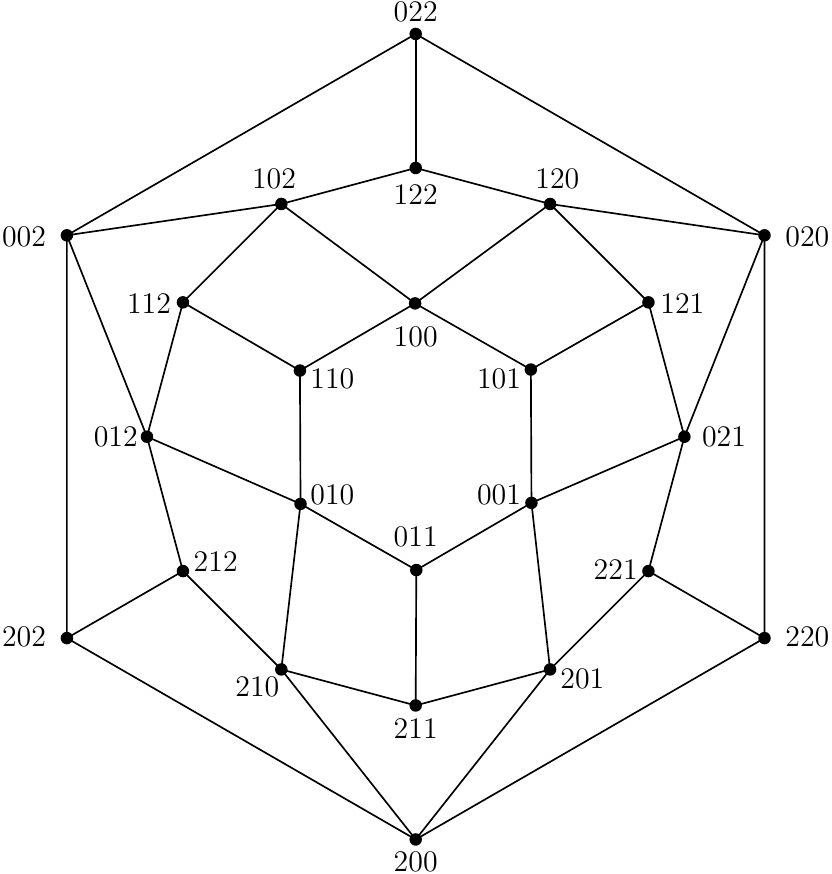}
  \end{center}
  \caption{The supertoken graph $F_{3\times 1}^2(S_2)$ as a planar graph with spectral radius $\rho\approx 3.5615$.}
  \label{fig:supertoken(maco)}
  \end{figure}
\begin{figure}[h!]
    \begin{center}
  \includegraphics[width=9cm]{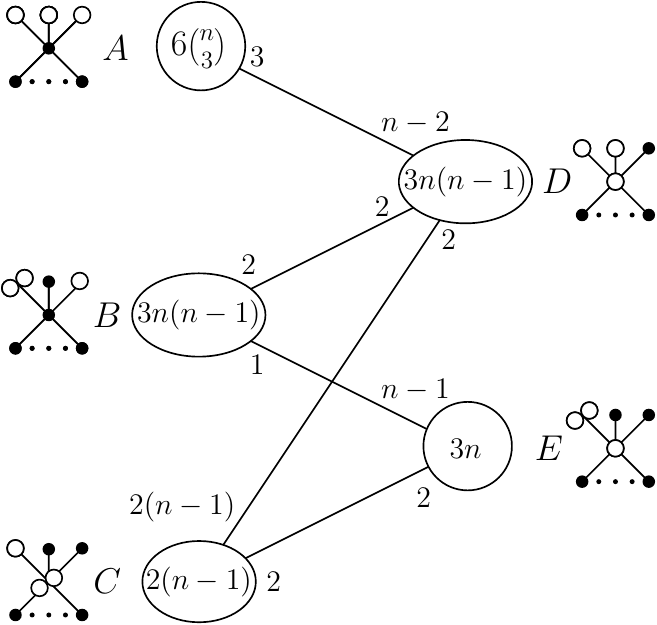}
  \end{center}
  \caption{The quotient diagram of the regular partition of the supertoken graph $F_{3\times 1}^2(S_n)$.}
  \label{fig:Quotient2}
  \end{figure}
Then, the quotient matrix is
$$
\Q= \left( 
   \begin{array}{ccccc}
     0 & 0 & 0   &  3  & 0  \\
     0 & 0 & 0   & 2   & 1  \\
     0 & 0 & 0   & 2(n-1)  & 2 \\
     n-2 & 2 & 2 & 0   & 0 \\
     0 & n-1 & 2   & 0 & 0  
   \end{array}
   \right),\
$$
with the eigenvalues claimed in \eqref{eigenvals:supertoken}. Moreover, the eigenvalue $\rho$ has a positive eigenvector, so that, by the Perron-Frobenius theorem, it corresponds to the spectral radius.
\end{proof}
For instance, for the case $n=2$, we obtain
\begin{align*}
\spec \Q &= \left\{\pm\frac{1} {2}(\sqrt{17}+3),
\pm\frac{1} {2}(\sqrt{17}-3),0
\right\}\\
&\subset \
\spec F_{3\times 1}^{2}(S_2)
=\left\{\pm\frac{1} {2}(\sqrt{17}+3), \pm2\sqrt{2}^2,\pm\sqrt{5}, \pm\sqrt{2}^4
\pm \frac{1}{2}(\sqrt{17}-3),0^6\right\},
\end{align*}
whereas, for the case $n=3$, we get
\begin{align*}
\spec \Q &= \left\{\pm 4.4838,
\pm 0.9462,0
\right\}\\
&\subset \
\spec F_{3\times 1}^{2}(S_3)
=\left\{\pm 4.4838^1,\pm 3.4641^4,\pm 3.2004^2,\pm 1.7320^{11},\pm 1.3256^2,\pm0.9462^1, 0^{18}\right\}.
\end{align*}
Notice that, in both cases, Proposition \ref{propo:F_3x1^2(S_n)} also gives the eigenvalues different from 0 with minimum absolute value.

\section{Conclusions}
In this paper, we investigate the spectra of supertoken graphs. More precisely, the existence of regular partitions has enabled us to give, either
the complete (or partial) spectrum of a supertoken graph of a star $S_n$.
In the first case, the whole spectrum has been given for the (super)token graphs $F_2(S_n)$ (Proposition \ref{prop:sp-F_2^1(S_n)}), $F_k(S_{2k-1})$ (Proposition \ref{propo:F_k(S_2k-1}), $F_2^2(S_n)$ (Proposition \ref{prop:sp-F_2^2(S_n)}), and $F_{2\times 1}^1(S_n)$
(Proposition \ref{prop:sp-F_2x1^1(S_n)}). 
In the second case, we provided some eigenvalues (such as the spectral radius) of $F_3^2(S_n)$
(Proposition \ref{propo:F_3^2(S_n)}) and $F_{3\times 1}^2(S_n)$ 
(Proposition \ref{propo:F_3x1^2(S_n)}).
Finally, Proposition
\ref{prop:regularF1_1111}
proves that $F_{k\times 1}^1(G)$ has a regular partition given by $F_k^1(G)$, so that all the eigenvalues of  $F_k^1(G)$ are also eigenvalues of $F_{k\times 1}^1(G)$ (Proposition \ref{propo:eigenvals F_{kx1}^1(G)}).

\end{document}